\documentclass[a4paper,twoside,12pt]{article}
\usepackage{amssymb,amsmath,amsthm,latexsym}
\usepackage{amsfonts}
\usepackage{graphicx,caption,subcaption}
\usepackage[pdftex,bookmarks,colorlinks=false]{hyperref}
\usepackage[hmargin=1.25in,vmargin=1.25in]{geometry}
\usepackage[auth-sc-lg,affil-sl]{authblk}
\def\ni{\noindent}
\def\N{\mathbb{N}_0}
\def\S {\Sigma}
\def\s {\sigma}
\def\cP{\mathcal{P}}

\newtheorem{thm}{Theorem}[section]

\newtheorem{cor}[thm]{Corollary}
\newtheorem{defn}[thm]{Definition}

\newtheorem{lem}[thm]{Lemma}

\newtheorem{prob}{Problem}
\newtheorem{prop}[thm]{Proposition}

\title{\textbf{\sc A Study on Integer Additive Set-Valuations of Signed Graphs}}

\author{N. K. Sudev}
\affil{\small Department of Mathematics\\ Vidya Academy of Science \& Technology \\ Thalakkottukara, Thrissur - 680501, Kerala, India.\\ E-mail: sudevnk@gmail.com}

\author{K. A. Germina}
\affil{\small Mathematics Research Centre\\ Mary Matha Arts \& Science College\\Mananthavady, Wayanad-670645, Kerala, India.\\ E-mail: srgerminaka@gmail.com}

\date{}

\begin{document}
\maketitle

\begin{abstract}
Let $\N$ denote the set of all non-negative integers and $\cP(\N)$ be its power set. An integer additive set-labeling (IASL) of a graph $G$ is an injective set-valued function $f:V(G)\to \cP(\N)-\{\emptyset\}$ such that the induced function $f^+:E(G) \to \cP(\N)-\{\emptyset\}$ is defined by $f^+ (uv) = f(u)+ f(v)$, where $f(u)+f(v)$ is the sumset of $f(u)$ and $f(v)$. A graph which admits an IASL is usually called an IASL-graph. An IASL $f$ of a graph $G$ is said to be an integer additive set-indexer (IASI) of $G$ if the associated function $f^+$ is also injective.  In this paper, we define the notion of integer additive set-labeling of signed graphs and discuss certain properties of signed graphs which admits certain types of integer additive set-labelings.
\end{abstract}

\vspace{0.2cm}

\ni \textbf{Key words}: Signed graphs; balanced signed graphs; clustering of signed graphs; integer additive set-labeled signed graphs; strong integer additive set-labeled signed graphs; weak integer additive set-labeled signed graphs; isoarithmetic integer additive set-labeled signed graphs.

\vspace{0.04in}

\section{Introduction}

For all  terms and definitions, not defined specifically in this paper, we refer to \cite{JAG,FH} and  and for the topics in signed graphs we refer to \cite{TZ1,TZ2}. Unless mentioned otherwise, all graphs considered here are simple, finite and have no isolated vertices.

\subsection{An Overview of IASL-Graphs}

The {\em sum set} (see \cite{MBN}) of two sets $A$ and $B$, denoted by $A+B$, is defined as $A+B=\{a+b:a\in A, b\in B\}$.  Let $\mathbb{N}_0$ be the set of all non-negative integers and let $X$ be a non-empty subset of $X$. Using the concepts of sumsets, we have the following notions as defined in \cite{GS1,GS0}.

An {\em integer additive set-labeling} (IASL, in short) is an injective function $f:V(G)\to \cP(X)-\{\emptyset\}$ such that the induced function $f^+:E(G)\to \cP(X)-\{\emptyset\}$ is defined by $f^+{uv}=f(u)+f(v)~ \forall uv\in E(G)$.  A graph $G$ which admits an IASL is called an {\em integer additive set-labeled graph} (IASL-graph).  

An {\em integer additive set-indexer} (IASI) is an injective function $f:V(G)\to \cP(X)-\{\emptyset\}$ such that the induced function $f^+:E(G) \to \cP(X)-\{\emptyset\}$ is also injective. A graph $G$ which admits an IASI is called an {\em integer additive set-indexed graph} (IASI-graph).

An IASL (or IASI) is said to be {\em $k$-uniform} if $|f^+(e)| = k$ for all $e\in E(G)$. That is, a connected graph $G$ is said to have a $k$-uniform IASL (or IASI) if all of its edges have the same set-indexing number $k$. The cardinality of the set-label of an element (vertex or edge) of a graph $G$ is called the {\em set-indexing number} of that element. If the set-labels of all vertices of $G$ have the same cardinality, then the vertex set $V(G)$ is said to be {\em uniformly set-indexed}. An element is said to be {\em mono-indexed} if its set-indexing number is $1$.

A {\em weak integer additive set-labeling} of a graph $G$ is an IASI $f:V(G)\to \cP(X)-\{\emptyset\}$ such that $|f^+(uv)| = \max(|f(u)|,|f(v)|)$ for all $u, v \in V(G)$ and a {\em strong integer additive set-indexer} (SIASI) of $G$ is an IASI such that if $|f^+(uv)| = |f(u)|\,|f(v)|$ for all $u, v \in V(G)$. 

\vspace{0.2cm}

The following result is a necessary and sufficient condition for a graph to admit a weak IASL.

\begin{lem}\label{L-WIASLG1}
{\rm \cite{GS3}} An IASI $f$ of a given graph $G$ is a weak IASI of $G$ if and only if at least one end vertex of every edge of $G$ is mono-indexed, with respect to $f$.
\end{lem}

\begin{thm}
{\rm \cite{GS0}} A graph $G$ admits a weakly uniform IASL if and only if $G$ is bipartite. 
\end{thm}

If a graph $G$ has a integer additive set-indexer $f:V(G)\to \cP(X)-\{\emptyset\}$ such that $|f^+(uv)|=|f(u)+f(v)|=|f(u)|\,|f(v)|$ for all edges $uv$ of $G$, then $f$ is said to be a {\em strong IASI} of $G$. A graph which admits a strong IASI is called a {\em strong IASI-graph}.

\begin{thm}\label{T-SIAIS1}
{\rm \cite{GS2}} A graph $G$ admits a strong IASI, say $f$, if and only if for any two adjacent vertices in $G$, the sets defined by $D_{f(u)}=\{|a-b|:a,b \in f(u)\}$ and $D_{f(v)}=\{|c-d|:c,d \in f(v)\}$ are disjoint.
\end{thm}

\begin{thm}
{\rm \cite{GS2}} A connected graph $G$ admits a strongly $k$-uniform IASL if and only if either $G$ is bipartite or $k$ is a perfect square.
\end{thm}

An IASL $f$ of a given graph $G$ is called an arithmetic IASL of $G$ if the elements of the set-labels of the vertices and edges of $G$ are in arithmetic progressions. If all these arithmetic progressions have the same common difference $d$, then such an arithmetic IASL is called \textit{isoarithmetic IASL} of $G$.

\begin{thm}
{\rm \cite{GS4}} If $f:V(G)\to \cP(X)$ is an isoarithmetic IASL defined on a graph $G$, then the cardinality of the set-label of any edge $uv$ in $G$  is $|f(u)|+|f(v)|-1$.
\end{thm}

\subsection{Preliminaries on Signed Graphs}

Note that a half edge of a graph $G$ is an edge having only one end vertex and a loose edge of $G$ is an edge having no end vertices. 

A \textit{signed graph} (see \cite{TZ1,TZ2}), denoted by $\S(G,\s)$,  is a graph $G(V,E)$ together with a function $\s:E(G)\to \{+,-\}$ that assigns a sign, either $+$ or $-$, to each ordinary edge in $G$. The function $\s$ is called the {\em signature} or {\em sign function} of $\S$, which is defined on all edges except half edges and is required to be positive on free loops.

An edge $e$ of a signed graph $\S$ is said to be a \textit{positive edge} if $\s(e)=+$ and an edge $\s(e)$ of a signed graph $\S$ is said to be a \textit{negative edge} if $\s(e)=-$. The set $E^+$ denotes the set of all positive edges in $\S$ and the set $E^-$ denotes the set of negative edges in $\S$. A  simple  cycle (or path) of a signed graph $\S$  is said to be {\em balanced} (see \cite{AACE,FHS}) if the product of signs of its edges is $+$. A  signed  graph  is said to be a {\em balanced signed graph} if it contains no half edges and all of its simple cycles are balanced. 

It is to be noted that the number of all negative  signed graph is balanced if and only if it is bipartite. 

Balance or imbalance is the fundamental property of a signed graph. The following theorem, popularly known as {\em Harary's Balance Theorem}, establishes a criteria for balance in a signed graph.

\begin{thm}
{\rm \cite{FHS}} The following statements about a signed graph are equivalent.
\begin{enumerate}\itemsep0mm
\item[(i)] A signed graph $\S$ is balanced.
\item[(ii)] $\S$ has no half edges and there is a partition $(V_1,V_2)$ of $V(\S)$ such that $E^-=E(V_1,V_2)$.
\item[(iii)] $\S$ has no half edges and any two paths with the same end points have the same sign. 
\end{enumerate}
\end{thm}

A signed graph $\S$ is said to be  \textit{clusterable} or \textit{partitionable} (see \cite{TZ1,TZ2}) if its vertex set can be partitioned into subsets, called \textit{clusters}, so that every positive edge joins the vertices within the same cluster and every negative edge joins the vertices in the different clusters. If $V(\S)$ can be partitioned in to $k$ subsets with the above mentioned conditions, then the signed graph $\S$ is said to be \textit{$k$-clusterable}. In this paper, we discuss only the $2$-clusterability of signed graphs.

It can be noted that $2$-clusterability always implies balance in a signed graph $\S$. The converse need not be true. If all edges in $\S$ are positive edges, then $\S$ is balanced but not $2$-clusterable.

In this paper,  we extend the studies on different types of integer additive set-labeling of graphs to classes of signed graphs and hence study the properties and characteristics of such signed graphs.

\section{IASL-Signed Graphs}

Motivated from the studies on set-valuations of signed graphs in \cite{BDAS}, and the studies on integer additive set-labeled graphs in \cite{GS1,GS0,GS2,GS4}, we define the notion of an integer additive set-labeling of signed graph as follows.

\begin{defn}{\rm
Let $X\subseteq \N$ and let $\S$ be a signed graph, with corresponding underlying graph $G$ and the signature $\s$. An injective function $f:V(\S)\to \cP(X)-\{\emptyset\}$ is said to be an \textit{integer additive set-labeling} (IASL) of $\S$ if $f$ is an integer additive set-labeling of the underlying graph $G$ and the signature of $\S$ is defined by $\s(uv))=(-1)^{|f(u)+f(v)|}$. A signed graph which admits an integer additive set-labeling is called an \textit{integer additive set-labeled signed graph} (IASL-signed graph) and is denoted by $\S_f$. }
\end{defn}

\begin{defn}{\rm 
An integer additive set-labeling $f$ of a signed graph $\S$ is said to be an integer additive set-indexer of $\S$ if $f$ is an integer additive set-indexer of the underlying graph $G$.}
\end{defn}

\begin{defn}{\rm 
An IASL $f$ of a signed graph $\S$ is called a \textit{weak IASL} or a \textit{strong IASL} or an \textit{arithmetic IASL} of $\S$, in accordance with the IASL $f$ of the underlying graph $G$ is a weak IASL or a strong IASL or an arithmetic IASL of the corresponding underlying graph $G$. }
\end{defn}

The structural properties and characteristics of different types IASL-signed graphs are interesting. In the following section, we study the properties of strong IASL-signed graphs. 

\subsection{Strong IASL-Signed Graphs}

As stated earlier, balance is the fundamental characteristic of a signed graph and hence let us investigate the conditions required for a strong IASL-signed graph to have the balance property.

The following result provides a necessary and sufficient condition for the existence of a balanced signed graph corresponding to a strongly uniform IASL-graph. 

\begin{thm}\label{P-BSG-SU}
A strongly $k$-uniform IASL-signed graph $\S$ is balanced if and only if the underlying graph $G$ is a bipartite graph or $\sqrt{k}$ is an even integer. 
\end{thm}
\begin{proof}
Assume that the  strongly $k$-uniform IASL-signed graph $\S$ is balanced. Then, for any cycle $C_r$, $\s(C_r)$ must be positive. Let $n_1$ and $n_2$ be two positive integers greater than $1$ such that $n_1n_2=k$. Label, the vertices of $C_r$ alternatively by $n_1$-element subsets and $n_2$-element subsets of the ground set $X$. Here we have the following cases.

\ni {\small \bf Case-1:} Let $n_1\ne n_2$. We claim that this labeling is possible only when $C_r$ is even. For, if $C_r$ is an odd cycle, then on labeling the vertices of $C_r$ as mentioned above, there will be two adjacent vertices, say $u$ and $v$ both having $n_1$-element set-labels or $n_2$-element set-labels and the edge $uv$ has the set-indexing number $n_1^2$ (or $n_2^2)$), which is a contradiction to the fact that $G$ is strongly $k$-uniform IASL-graph. Therefore, $G$ is bipartite.

\ni{\small \bf Case-2} Let $C_r$ be an odd cycle in $G$. Then, we claim that the above mentioned labeling is possible only when $n_1=n_2$. For, if $C_r$ is an odd cycle, as mentioned in Case-1, there exists an edge in $C_r$ with set-indexing number $n_1^2$ (or $n_2^2$). Since $G$ admits a strongly $k$-uniform IASL, we have $n_1^2=k=n_1n_2$. This is true only if $n_1=n_2=\sqrt{k}$. That is, $k$ is a perfect square.  Therefore, every vertex of $C_r$ has the set-indexing number $\sqrt{k}$ and every edge of $C_r$ has the set-indexing number $k$. Since $\S$ is balanced, we have $\s(C_r)=+$, which is possible when $k$ and hence $\sqrt{k}$ are even integers.  

Conversely, assume that the underlying graph $G$ of strongly $k$-uniform IASL-signed graph $\S$ is a bipartite graph or $\sqrt{k}$ is an even integer. Then, consider the following cases.

\ni {\small \bf Case-1:} Assume that the underlying graph $G$ is a bipartite graph. Then, $G$ has no odd cycles. Let $C_r$ be an arbitrary cycle in $G$, where $r$ is an even integer. Consider $n_1,n_2\in \N$ such that $n_1n_2=k$. Now, label the vertices of $G$ alternatively by $n_1$-element subsets and $n_2$ element subsets of the ground set $X$ such that the labeling becomes a strongly $k$-uniform IASL of $G$. Here, we have the following subcases.

\ni {\small \bf Subcase-1.1} If both $n_1$ and $n_2$ are odd integers, then $k$ is odd and hence every edges of the cycle $C_r$ in $\S$ has the negative sign. Since, $C_r$ has even number of edges, we have $\s(C_r)=+$. 

\ni {\small \bf Subcase-1.2} If one or both of $n_1$ and $n_2$ is even, then $k$ is even and every edge of $C_r$ has the positive sign. Therefore, $\s(C_r)=+$.

\ni {\small \bf Case-2} Assume that $G$ is a non-bipartite graph. By hypothesis, $\sqrt{k}$ is an even integer and hence $k$ is also an even integer. Since $G$ is not bipartite, it contains odd cycles. Let $C_r$ be an arbitrary odd cycle in $G$. Since $G$ admits a strongly $k$-uniform IASL, every edge of $C_r$ must be labeled by the subsets of $X$ having cardinality $\sqrt{k}$ (see \cite{GS2}). Therefore, every edge of $G$ has positive sign and hence $\s(C_r)=+$. 

In all the above cases, we can see that the strongly $k$-uniform IASL-signed graph $\S$ is balanced.
\end{proof}

What are the conditions required for a strongly uniform IASL-signed graph to be clusterable? The following result provides a solution to this problem.

\begin{prop}
A strongly $k$-uniform IASL-signed graph $\S$ is clusterable if and only if the underlying graph $G$ is bipartite and $k$ is an odd integer.
\end{prop}
\begin{proof}
Let $\S$ is clusterable. Let $(U_1,U_2)$ be partition of $V(\S)$ with the required properties of a clustering of $\S$. Clearly, $k$ must be odd. For, if $k$ is even all edges of $\S$ will be positive edges and hence all vertices of $\S$ belong to either $U_1$ or to $U_2$ making other empty, contradicting the fact that $\S$ is clusterable. As $k$ is odd, every edge of $\S$ is a negative edge and hence for any two adjacent vertices in $\S$ must belong to different partitions. Choose vertices which are pairwise non-adjacent in $\S$ to form a subset $U_1$ of $V(\S)$ and Let $U_2=V-U_1$. Clearly, $U_2$ is also a subset of $V$ in which vertices are pairwise disjoint. Therefore, $(U_1,U_2)$ is a bipartition of the underlying graph $G$. Hence $G$ is bipartite.

Conversely, assume that the underlying graph $G$ of a strongly $k$-uniform IASL-signed graph $\S$ is  bipartite graph with bipartition $(V_1,V_2)$ and $k$ is an odd integer. Therefore, every edge of $\S$ is a negative edge with one end in $V_1$ and other end in $V_2$. Therefore, $(V_1,V_2)$ satisfies the properties of a clustering for $\S$. Hence, $\S$ is clusterable.
\end{proof}

If the underlying graph $G$ of a strong IASL-signed graph $\S$ is bipartite, then $\S$ is balanced if and only if the number of negative edges in $\S$ in every cycle of $G$ must be even. This is possible only when the number of distinct pairs of adjacent vertices, having odd parity set-labels, in every cycle of $\S$ must be even. Therefore, we have 

\begin{prop}
Let $\S$ be a strong IASL- signed graph with the underlying graph $G$ is bipartite. Then, $\S$ is clusterable if and only if the number of distinct pairs of adjacent vertices having odd parity set-labels is even.
\end{prop}  

The proof of the above theorem is very obvious. The following result describes the conditions required for the clusterability of (non-uniform) strong IASL-signed graphs whose underlying graph $G$ is a bipartite graph.

\begin{prop}\label{P-CSG-S1}
The strong IASL-signed graph, whose underlying graph $G$ is a bipartite graph, is clusterable if and only if there exist at least two adjacent vertices in $\S_f$ with odd parity set-labels.
\end{prop}
\begin{proof}
Let $\S$ be a strong IASL-signed graph whose underlying graph $G$ is a bipartite graph. Then, the same IASL of $\S$ is a strong IASL of $G$ also. Since $G$ is bipartite, every cycle in $G$ is an even cycle. Let $C_r:v_1v_2v_3\ldots v_rv_1$ be a cycle in $G$. 

First assume that $\S$ is clusterable. Then, there exists a partition $(U_1,U_2)$ of non-empty subsets of $V(\S)$ such that the edges connecting vertices in the same partition have positive sign and the edges connecting vertices in the different partitions have the negative sign. Note that an edge $uv$ of $G$ has a negative sign only when both $u$ and $v$ have odd parity set-labels. Since $G$ is a connected graph and both sets $U_1$ and $U_2$ are non-empty, there must be at least one edge, say $e=uv$, in $G$ with one end vertex in $U_1$ and the other end vertex in $U_2$ such that both $u$ and $v$ have odd parity set-labels.

Conversely, assume that at least two adjacent vertices of $G$ have odd parity set-labels. If $u$ and $v$ be two vertices of $G$ having odd cycles in $G$. Then, $\s(uv)=-$ in $\S$. Let $U_1$ and $U_2$ be two mutually exclusive subsets of $V(G)$ such that $u\in U_1$ and $v\in U_2$. If there exist other edges, say $xy$ such that both $x$ and $y$ have odd parity set-labels, then include any one of $x$ and $y$ to $U_1$ and all other vertices to $U_2$. Repeat this process until all adjacent pairs of vertices having odd parity set-labels are counted. Then, $(U_1,U_2)$ will be a partition of $\S$ with desired properties. Therefore, $\S$ is clusterable.     
\end{proof}

Let $\S$ be a strong IASL- signed graph, whose underlying graph $G$ is a non-bipartite graph. Then $G$ contains some odd cycles. If $C_r$ is an arbitrary odd cycle in $G$, then $\S$ is balanced if and only if the number of negative edges in $C_r$ is even, which is possible only when the number of positive edges in $C_r$ is odd. It is possible only when at least two adjacent vertices must have even parity set-labels. 


A necessary and sufficient condition for a strong IASL-signed graph to be clusterable is described in the following theorem.

\begin{thm}\label{T-CSG2}
A strong IASL-signed graph $\S$ is clusterable if and only if every odd cycle of the underlying graph $G$ has at least two adjacent vertices with even parity set-label and at least two adjacent vertices with odd parity set-label. 
\end{thm}
\begin{proof} 
Let $\S$ be a strong IASL-signed graph with the underlying graph $G$, where $G$ is a non-bipartite graph. Then, $G$ contains odd cycles. Let $C_r:v_1v_2v_3\ldots v_rv_1$ be a cycle of length $r$ in $G$. 

If $\S$ is clusterable, there exists a partition of vertices $(U_1,U_2)$ such that all edges having end vertices in the same partition have positive sign and the edges having end vertices in the different partitions have negative sign. If all vertices of $\S$ have even parity set-labels, then all edges of $\S$ will be positive edges. Hence, all vertices of $\S$ must belong to the same partition, say $U_1$, making the other partition, say $U_2$ empty. If one end vertex of every edge of $\S$ has even parity set-label, then also all edges of $\S$ become positive edges. In this case also, all vertices of $\S$ are in the same partition and the other partition is empty. Hence, there must be at least one edge in $\S$ such that its both end vertices have odd parity set-labels. 

Conversely, assume that every odd cycle, say $C_r$, contains two adjacent vertices having even parity set-labels. Without loss of generality, let $v_1$ and $v_2$ be the vertices in $C_r$, which have even parity set-labels. Let $v_1,v_2\in U_1$. Let all other vertices have odd parity set-labels. Since $\s(v2v_3)=+$, $v_3$ must also be an element of $U_1$.  Since $\s(v_3v_3)=-$, $v_3\in U_2$. Proceeding like this, we have $v_4,v_6,\ldots,v_{r-1}$ are in $U_2$ and $v_5,v_7,\ldots,v_r$ are in $U_1$. This partition $(U_1,U_2)$ is a clustering for $\S$. 
\end{proof}

In this context, the questions on the balance and clusterability of weak IASL-signed graphs arouse much interest. In the following section, we discuss certain properties of weak IASL-signed graphs that are similar to those of strong IASL-signed graphs.

\subsection{Weak IASL-Signed Graphs}

 Balance and clusterability of the induced signed graphs of weak IASL-graphs has been described in the following theorems. 

Analogous to Proposition \ref{P-BSG-SU}, the balance of weakly uniform IASL-signed graph can be described as follows. 

\begin{prop}\label{P-BSG-WU}
A weakly $k$-uniform IASL-signed graph is always balanced.
\end{prop}
\begin{proof}
Let $\S$ be a weakly $k$-uniform IASL-signed graph with the underlying graph $G$, where $k$ is any positive integer greater than $1$. Then, $G$ admits a weakly $k$-uniform IASL, say $f$, then $f^+(e)=k$ for all $e\in G$. Then, the signature $\s$ is given by $\s(e)=(-1)^k$ for all $e\in \S$. Therefore, the signs of all edges of $\S$ are all odd or all even.  Since the underlying graph $G$ is a weakly $k$-uniform IASL-graph, then $G$ is bipartite (see \cite{GS1}). Therefore, $G$ has no odd cycles. Therefore, the number of signs, positive or negative, of edges in each cycles are even. Therefore, for any cycle $C_r$ in $\S$, $\s(C_r)$ is positive. Hence, $\S$ is balanced.
\end{proof}

\ni The following theorem discusses the clusterability of weakly uniform IASL-signed graphs.

\begin{thm}
A weakly $k$-uniform IASL-signed graph $\S$ is clusterable if and only if $k$ is a positive odd integer. 
\end{thm}
\begin{proof}
Let the given weakly $k$ uniform IASL-signed graph $\S$ be clusterable. Then, there exists a partition $(U_1,U_2)$ of non-empty subsets of $V(\S)$ such that the end vertices of positive edges belongs to the same partition and the end vertices of negative edges belong to different partitions. If $k$ is even, then all edges in $\S$ are positive edges and all vertices in $\S$ belong to the same partition, say $U_1$. Therefore, $U_2=\emptyset$, which contradicts the hypothesis that $\S$ is clusterable. Hence, $k$ is an odd integer.

Conversely, assume that $k$ is an odd integer. Then, every edge of $G$ is a negative edge. Then, the bipartition $(V_1,V_2)$ of the underlying graph $G$, where $V_1$ is the set of all mono-indexed vertices and $V_2$ is the set of all vertices having set-indexing number $k$, will form a $2$-clustering of $\S$. That is, $\S$ is clusterable.
\end{proof}

Balance of a weak IASL-signed graph whose underlying graph is a bipartite graph is discussed in the following result.

\begin{prop}
A weak IASL-signed graph $\S$, whose underlying graph $G$ is a bipartite graph, is balaced if and only if the number of odd parity non-singleton set-labels in every cycle of $\S$ is even.
\end{prop}
\begin{proof}
Assume that a weak IASL-signed graph $\S$ is balanced. Note that for the corresponding underlying graph $G$, the set-indexing number of every edge, not mono-indexed, is the cardinality of the non-singleton set-label of its end vertex. Hence, for every odd parity non-singleton vertex set-labels in $\S$, the corresponding edge has a negative sign. Hence, for any cycle $C_r$ in $\S$ we have  $\s(C_r)=+$ and this is possible only when the number of odd parity non-singleton vertex set-labels in $C_r$ of $\S$ must be even.

Conversely, assume that the number of odd parity non-singleton vertex set-labels in any cycle of $\S$ is even. Therefore, the number of negative edges in $\S$ is even. Hence, for any cycle $C_r$ in $\S$, $\s(C_r)=+$ and hence $\S$ is balanced.
\end{proof}

The following theorem establishes a necessary and sufficient condition for a weak IASL-signed graph whose underlying graph is a bipartite graph.

\begin{thm}\label{T-CSG1}
The weak IASL-signed graph $\S$, whose underlying graph $G$ is a bipartite graph, is clusterable if and only if there exist some non-singleton vertex set-labels of $\S$ are of odd parity.
\end{thm}
\begin{proof}
Let $G$ is a bipartite graph with bipartition $(V_1,V_2)$. Then $G$ need not have any mono-indexed edge. Then, without loss of generality, let $V_1$ contains all mono-indexed vertices and $V_2$ contains all vertices having non-singleton set-labels. Let $\S$ denotes the corresponding induced signed graph $\S$ of $G$. 

Assume that some set-labels of the vertices in $V_2$ are of odd parity. Since the mono-indexed vertices in $G$ are not adjacent in $G$, all vertices in $V_1$ can be in the same cluster $U_1$, if exists. Then, by the definition of clustering, the vertices having even parity set-labels cannot be included in the second cluster $U_2$ as the signs of edges connecting these vertices to the vertices in $V_1$ are positive.  Therefore, let $U_1=V_1\cup V_2^{\prime}$ and $U_2=V_2-V_2^{\prime}$, where $V_2^{\prime}$ is the proper subset of all vertices in $V_2$ having even parity set-labels. Clearly, all the edges in the same partition, if exist, have positive signs and the edges connecting the vertices in different partitions have negative sign. That is, $G$ is clusterable. 

Conversely, assume that $\S$ is clusterable. Then, there exists a partition $(U_1, U_2)$ of the vertex set $V(\S)$ such that all edges connecting the vertices in the same partition have the positive sign and the edges connecting the vertices in different partitions have negative sign. Let $U_1$ contains all vertices in $V_1$. Any vertex $u$ in $V_2$, having an even parity set label and adjacent to some vertex $v$ in $V_1$ must also belong to $U_1$ as $\s(uv)=+$. Hence, if the set-labels of all vertices in $V_2$ are of even parity, then $U_2=\emptyset$, which is a contradiction to the hypothesis that $\S$ is clusterable. Therefore, the set-labels of some vertices in $V_2$ are of odd parity.
\end{proof}

In this context, it is much interesting to check the balance property of weak IASL-signed graphs whose underlying graphs are non-bipartite. Hence, we have the following theorem.

\begin{thm}\label{Thm-2.13}
A weak IASL-signed graph $\S$, whose underlying graph $G$ is a non-bipartite graph, is not balanced.  
\end{thm}
\begin{proof}
Since $G$ is a non-bipartite graph, $G$ contains odd cycles. Let $C_r$ be an odd cycle in $G$. If $\S$ is balanced, then the number of negative edges in $C_r$ must be even. When one vertex, say $v$, of $G$ has an even parity set-label, then the two edges incident on it will have the positive sign and the remaining odd number of edges in $C_r$ have negative signs. If $u$ and $v$ are two adjacent vertices in $C_r$, then the three edges incident on these two vertices become positive and the number of negative edges in $C_r$ becomes even. Therefore, if $G$ is balanced, then at least two adjacent vertices must have even parity set-labels. This contradicts the fact that $G$ admits a weak IASL. Hence, $\S$ is not balanced.        
\end{proof}

\ni The following result is a straight forward implication of the above theorem.

\begin{cor}
A weak IASL-signed graph $\S$, whose underlying graph $G$ is non-bipartite, is not clusterable.
\end{cor}
\begin{proof}
For any signed graph $\S$, we have $2$-clusterability implies the balance in $\S$. But by Theorem \ref{Thm-2.13}, a weak IASL-signed graph $\S$, whose underlying graph is non-bipartite, can not be a balanced signed graph. Hence, the weak IASL-signed graph $\S$ is not clusterable.
\end{proof}

Another interesting type IASL-signed graph is the signed graph which admits an isoarithmetic IASL. In the following section, we discuss the properties of these types of signed graphs.

\subsection{Isoarithmetic IASL-Signed Graphs}

The following theorem describes a necessary and sufficient condition for an isoarithmetic IASL-signed graph to be balanced.

\begin{thm}
An isoarithmetic IASL-signed graph $\S$ is balanced if and only if every cycle in $\S$ has even number of distinct pairs of adjacent vertices having the same parity set-labels.
\end{thm}
\begin{proof}
Let $\S$ be a an isoarithmetic IASL-signed graph. Then, for every edge $uv$ in $E(\S)$, the cardinality of the set-label of $uv$ is $|f^+(uv)|=|f(u)|+|f(v)|-1$. Therefore,  $|f^+(uv)|$ is odd if both $f(u)$ and $f(v)$ are of the same parity and $|f^+(uv)|$ is even if $f(u)$ and $f(v)$ are of different parities. 

Assume that $\S$ is balanced. Then, the number of negative edges in every cycle of $\S$ must be even, the number of disjoint pairs of adjacent vertices having the same parity set-labels must be even.

Conversely, assume that the number of disjoint pairs of adjacent vertices in every cycle $C_r$ of $\S$ having the same parity set-labels must be even. Then, the number of negative edges in $C_r$ is even. Therefore, $\S$ is balanced. 
\end{proof}

What are the conditions required for an isoarithmetic IASL-signed graph to be clusterable? The following result provides the required conditions in this regard.

\begin{prop}
An isoarithmetic IASL-signed graph $\S$ is clusterable if and only if $\S$ contains some disjoint pairs of adjacent vertices having the same parity set-labels.
\end{prop}
\begin{proof}
Note that a connected signed graph $\S$ clusterable, if and only if $\S$ must have negative edges connecting the vertices in different partitions. Hence, if an isoarithmetic IASL-graph $\S$ is clusterable, then $\S$ contains negative edges which is possible when some disjoint pairs of adjacent vertices in $\S$ must have the same parity set-labels. 
\end{proof}

\section{Conclusion}

In this paper, we discussed the characteristics and properties of the induced signed graphs of certain IASL-graphs with a prime focus on clusterability and balance of these signed graphs. There are several open problems in this area. Some of the open problems that seem to be promising for further investigations are following.

\begin{prob}{\rm 
Discuss the $k$-clusterability of different types of IASL-signed graphs for $k>2$.}
\end{prob}

\begin{prob}{\rm 
Discuss the balance and $2$-clusterability and general $k$-clusterability of other types of arithmetic IASL-signed  graphs of different types of arithmetic and semi-arithmetic IASL-graphs.}
\end{prob}

\begin{prob}{\rm 
Discuss the balance and $2$-clusterability and general $k$-clusterability of graceful, sequential and topological IASL-signed  graphs.}
\end{prob}

Further studies on other characteristics of signed graphs corresponding to different IASL-graphs are also interesting and challenging. All these facts highlight the scope for further studies in this area.

\end{document}